\documentclass[a4paper,12pt]{article}
\usepackage[T1]{fontenc}
\usepackage{lmodern}
\usepackage[utf8]{inputenc}
\usepackage{amsmath,amssymb,amsthm,mathrsfs}
\usepackage{indentfirst,enumitem,authblk}
\usepackage[numbers,square,sort&compress]{natbib}
\usepackage[hidelinks]{hyperref}
\usepackage{cleveref}
\numberwithin{equation}{section}
\allowdisplaybreaks
\newtheorem{theorem}{Theorem}[section]
\newtheorem{lemma}[theorem]{Lemma}
\newtheorem{thm}[theorem]{Theorem}

\theoremstyle{definition}

\theoremstyle{remark}
\newtheorem{rmk}[theorem]{Remark}

\begin{document}
\title{Invariant Densities on Cylinders with an Interface:\\
Existence, Uniqueness, and Exponential Asymptotics}
\author[a]{Pengxiu Yu\thanks{Email: pxyu@ruc.edu.cn}}
\author[b,c]{Yiping Zhang\thanks{Corresponding author, Email: zhangyiping161@mails.ucas.ac.cn}}
\affil[a]{\footnotesize School of Mathematics, Renmin University of China, Beijing 100872, China}
\affil[b]{\footnotesize School of Mathematics and Statistics, and Hubei Key Laboratory of Mathematical Sciences, Central China Normal University, Wuhan 430079, China}
\affil[c]{\footnotesize Key Laboratory of Nonlinear Analysis \& Applications (Ministry of Education), Central China Normal University, Wuhan 430079, China}
\date{}
\maketitle

\begin{abstract}
We study the stationary double-divergence equation on the infinite cylinder
$\mathbb R\times\mathbb T^{d-1}$, whose coefficients coincide with two
possibly different periodic media outside a bounded interface layer.
Under the centering conditions for the two periodic media, we construct
a positive invariant density that converges exponentially in $C^1$
at each end to a positive multiple of the corresponding periodic
invariant density. Every bounded
distributional solution is a scalar multiple of this density.
The two asymptotic amplitudes are determined by the full interface problem,
up to a common normalization. Motivated by the invariant-measure analysis of
Hairer and Manson [Ann. Probab. 39(2011) 648-682], our proof uses only PDE methods: a transformation based on the periodic invariant densities, adjoint Neumann problems on finite cylinders,
and energy estimates.
\end{abstract}


\section{Introduction}\label{s1}
\subsection{The stationary equation and its background}
Consider the following elliptic operator in nondivergence form:
$$\mathcal{L}u=a_{ij}\partial_{ij}u+b_i\partial_iu,\quad u\in C^\infty_0(\Omega),$$
where the leading coefficient $A=(a_{ij})$ is a matrix function defined on a domain $\Omega\subset\mathbb{R}^d$ with values in the space of positive symmetric linear operators on $\mathbb{R}^d$ and $b=(b_i)$ is a vector field on $\mathbb{R}^d$. For a Borel measure $\mu$ defined on $\Omega$, the adjoint equation is
$$\mathcal{L}^*\mu=0,$$
which is understood in the following sense:
$$\int_{\Omega}\mathcal{L}ud\mu=0,\quad \forall u\in C^\infty_0(\Omega).$$

Under suitable smoothness assumptions on the coefficients $A$ and $b$, the Borel measure $\mu$ admits a density $m$ with respect to Lebesgue measure such that $d\mu=mdx$ on $\Omega$, and the density $m$ satisfies the following double-divergence-form equation:
$$\partial_{ij}(a_{ij}m)-\partial_i(b_im)=0\quad \text{in }\Omega.$$

There is a vast literature \cite{MR765409,MR3694737,MR3747493,MR4887768} on the study of the
 equation in double-divergence form above, concerning existence, uniqueness, positivity and regularity of the solution. We refer the reader to the references therein for more results. For further background, see the books \cite{MR3443169, MR987631} and their references therein for more results concerning this topic.

From the viewpoint of stochastic analysis, the elliptic operator of nondivergence form can be viewed as the generator of the diffusion process determined by the following stochastic differential equation:
$$\left\{\begin{aligned}
dX_t&=b(X_t)dt+\sqrt{2}\sigma(X_t)dW_t,\\
X_0&=x,
\end{aligned}\right.$$
where $\sigma$ is the square root of the positive definite matrix $A=(a_{ij})$ and $W_t$ is a standard d-dimensional Wiener process. For more details, we refer to \cite{MR1121940} for readers' convenience. 

We now specialize the stationary equation to the infinite cylinder
\[
\mathbb D=\mathbb R\times\mathbb T^{d-1},\qquad
\mathbb T=\mathbb R/\mathbb Z,\qquad d\geq2,
\]
and write $y=(y_1,y')$. The equation to be studied is
\begin{equation}\label{1.1}
\mathcal L^*m
:=\partial_{ij}(a_{ij}m)-\partial_i(b_i m)=0
\qquad\text{in }\mathbb D.
\end{equation}
Here and below repeated indices are summed from $1$ to $d$.
The matrix $A=(a_{ij})$ is symmetric and uniformly elliptic.
Functions on $\mathbb D$ are identified with functions on $\mathbb R^d$
that are $1$-periodic in $y'$. We call this tangential periodicity
$\mathbb D$-periodicity, and use $\partial_i=\partial_{y_i}$.
The distributional formulation is the one given above, with tests in
$C_c^\infty(\mathbb D)$, that is, smooth tangentially periodic functions
with compact support in the axial variable.
A positive density satisfying this identity is called an invariant density.
Its associated measure is infinitesimally invariant for $\mathcal L$.

The starting point for our problem is the work of Hairer and Manson
\cite{MR2789509}. They considered diffusions whose coefficients are
periodic outside an interface region of finite thickness. After
identifying the directions parallel to the interface modulo
$\mathbb Z^{d-1}$, their state space becomes $\mathbb D$.
An essential ingredient is a $\sigma$-finite invariant measure,
unique up to scaling, which approaches positive multiples of the two
periodic invariant measures exponentially far from the interface;
see \cite[Proposition~5.5]{MR2789509}. Their analysis uses probabilistic
methods. They present the argument for constant diffusivity and
explicitly note its adaptation to nonconstant diffusivity.
This motivates a direct PDE study of the stationary equation
\eqref{1.1}, including the pointwise behavior of its density and gradient.
Moreover, by using the decay properties of the invariant measure obtained in \cite{MR2789509} and transferring the original homogenization problem of nondivergence form into a divergence-form equation whose coefficient matrix converges
exponentially to periodic matrices at the two ends, the second author \cite{MR4882925} obtained the size estimates of the Green function as well as its gradient for this oscillating operator of nondivergence form, which would imply the desired $O(\varepsilon)$ convergence rates.

Note that an earlier version of this work was posted as the arXiv preprint
\cite{YuZhang2025}, entitled \emph{Decay Properties of Invariant Measure
and Application to Elliptic Homogenization of Non-divergence Form with
an Interface}. That version imposed the additional assumption
$b_1\equiv0$ and included applications to homogenization.
The present version removes this restriction on the normal drift.
We have changed the title to reflect the focus on existence,
uniqueness, and exponential asymptotics for the stationary
double-divergence equation on the cylinder with an interface. The homogenization
applications from the earlier version are omitted here, since they follow directly from \cite[Section 4]{MR4882925} once the desired
decay properties have been established.

There is an established PDE theory for adjoint equations of this type.
Bauman \cite{MR765409} developed comparison and Harnack estimates for
positive solutions of nondivergence equations and their adjoints.
Bogachev, Krylov, and R\"ockner \cite{BogachevKrylovRockner2001}
proved the existence and local Sobolev regularity of densities for
solutions of elliptic equations for measures under weak coefficient
assumptions. Their subsequent work \cite{BogachevKrylovRockner2006}
established global Sobolev regularity and pointwise bounds for densities
under appropriate hypotheses on the coefficients.

Of particular relevance is the work of Bogachev and Shaposhnikov
\cite{MR3694737}, which directly addresses second-order elliptic
equations in double-divergence form. They obtained sharp conditions
for higher integrability and continuity of solutions with coefficients
of low regularity, as well as a Harnack inequality.
These results emphasize that the double-divergence structure requires
its own regularity theory when the leading coefficients are not
differentiable. In particular, the usual reduction to a
divergence-form equation with a drift involving derivatives of $A$
is not available at that level of regularity.
The questions of local regularity treated there complement the
global existence and uniqueness problem considered here.

Further PDE results treat several aspects of the same equation.
Dong, Escauriaza, and Kim \cite{MR3747493} derived adjoint Harnack
inequalities and boundary weak type-$(1,1)$ estimates under Dini mean
oscillation assumptions. Bogachev, R\"ockner, and Shaposhnikov
\cite{BogachevRocknerShaposhnikov2023} developed an analytic version
of Zvonkin's transformation. They proved Harnack inequalities with
Dini mean oscillation diffusion matrices and drifts in
$L^p_{\mathrm{loc}}$, $p>d$, and obtained existence results for
stationary probability solutions under additional global hypotheses.
For asymptotic analysis in a different geometric setting, Maz'ya and
McOwen \cite{MazyaMcOwen2007} described the leading behavior of
double-divergence solutions near an interior point where the
coefficients stabilize. More recently, Choi, Dong, Kim, and Kim
\cite{ChoiDongKimKim2025} proved differentiability on the zero level
set and pointwise derivative bounds for solutions with Dini mean
oscillation coefficients. Bogachev and Shaposhnikov
\cite{BogachevShaposhnikov2026} established solvability of a
double-divergence Dirichlet problem with Borel measures as boundary
data. These works concern local and global regularity, asymptotic
profiles, and boundary-value solvability; their
conclusions should be distinguished from uniqueness of a bounded
stationary density on a two-ended cylinder.

Existence and uniqueness also depend on the class of solutions and on
global assumptions. Bogachev and R\"ockner \cite{BogachevRockner2001}
proved existence of invariant probability measures under Lyapunov
conditions. Albeverio, Bogachev, and R\"ockner
\cite{AlbeverioBogachevRockner1999}, and Bogachev, R\"ockner, and
Stannat \cite{BogachevRocknerStannat2002}, studied uniqueness of
stationary probability solutions and its relation to diffusion
semigroups. Nonuniqueness examples in
\cite{BogachevRocknerStannat2002,Shaposhnikov2008} show why local
ellipticity and smoothness alone do not settle this question.
See also \cite{BogachevRocknerShaposhnikov2011,MR3443169}.
For the cylinder problem below, we seek a bounded density with
positive periodic limits at both ends, rather than a finite-mass
probability solution. The issue is to construct and characterize
such a density when the two periodic media are joined through
an arbitrary bounded interface layer.

A closely related analytical framework was developed by Pettersson and
Piatnitski \cite{PetterssonPiatnitski2018} for stationary
convection--diffusion equations in infinite cylinders with lateral
Neumann conditions and different periodic regimes at the two ends.
Their results relate solvability to the effective axial drifts.
In the centered case, they also construct a positive adjoint solution,
unique up to multiplication, that stabilizes to periodic regimes;
see \cite[Theorem~4.1]{PetterssonPiatnitski2018}.
Our analysis is consistent with this approach. We work on a cylinder
with torus cross-section and formulate the result for a double-divergence
equation, with exponential estimates for both the density and its
gradient and a uniqueness statement in the full class of bounded,
possibly signed, distributional solutions.

\subsection{Assumptions and main result}
Fix $\beta\in(0,1)$ and $0<\lambda\leq\Lambda$. We assume
\begin{equation}\label{1.2}
\begin{gathered}
A=A^*,\qquad
\lambda|\xi|^2\leq A(y)\xi\cdot\xi\leq\Lambda|\xi|^2
\quad(y\in\mathbb D,\ \xi\in\mathbb R^d),\\
A\in C^{1,\beta}(\mathbb D),\qquad
b\in C^{0,\beta}(\mathbb D).
\end{gathered}
\end{equation}
The H\"older norms here are global, ${A}^*$ denotes the transpose of ${A}$ and the coefficients are
$\mathbb D$-periodic. Let $A_\pm$ and $b_\pm$ be $1$-periodic in
every coordinate, with the same ellipticity and regularity, such that
\begin{equation}\label{1.3}
A(y)=A_+(y)\quad(y_1>1),\qquad
A(y)=A_-(y)\quad(y_1<-1),
\end{equation}
and
\begin{equation}\label{1.4}
b(y)=b_+(y)\quad(y_1>1),\qquad
b(y)=b_-(y)\quad(y_1<-1).
\end{equation}
Within $[-1,1]\times\mathbb T^{d-1}$, the coefficients are arbitrary
subject to \eqref{1.2}. Thus the interface is a bounded transition
layer.

Write
\begin{equation}\label{1.5}
\mathcal L_\pm u=a_{\pm,ij}\partial_{ij}u+b_{\pm,i}\partial_i u\quad \text{and}\quad \mathcal L u=a_{ij}\partial_{ij}u+b_{i}\partial_i u.
\end{equation}
The classical theory on the compact torus gives a unique normalized
positive periodic density $m_\pm$:
\begin{equation}\label{1.6}
\begin{cases}
\partial_{ij}(a_{\pm,ij}m_\pm)-\partial_i(b_{\pm,i}m_\pm)=0
    \quad\quad \text{in }\mathbb Y\cong \mathbb R^d/\mathbb Z^d,\\
\displaystyle\int_{\mathbb Y}m_\pm=1,\qquad
m_\pm\in C^{1,\beta}(\mathbb Y),\qquad
0<c_*\leq m_\pm\leq C_*.
\end{cases}
\end{equation}
See \cite{MR3443169} for the theory of stationary densities. In fact, it directly follows from \cite[Proposition 2.1]{MR4657308} that $m_\pm\in C^{0,\alpha}(\mathbb{Y})$ for some $\alpha\in (0,1)$. Now, by rewriting the equation \eqref{1.6} as 
\begin{equation*}
\partial_{i}\left({a}_{\pm, ij} \partial_j m_\pm\right)=\partial_{i}\left(b_{\pm,i}m_\pm-\partial_j {a}_{\pm, ij} \cdot m_\pm \right)\text{ in }\mathbb{Y},\end{equation*}
and by using the $C^{1,\alpha}$ theory twice, we know that $m_\pm\in C^{1,\beta}(\mathbb{Y})$.
Note that we also need that the drifts $b_\pm$ satisfy the so-called centering conditions:
\begin{equation}\label{1.7}
\int_{\mathbb Y}b_{\pm,i}m_\pm=0,\qquad i=1,\ldots,d.
\end{equation}

As pointed out earlier, we aim to seek a positive solution of \eqref{1.1} that remains uniformly
bounded above and away from zero throughout $\mathbb D$, and has the
prescribed type of periodic behavior at both ends:
\[
m(y)\sim q_\pm m_\pm(y)\qquad(y_1\to\pm\infty).
\]
The positive amplitudes $q_\pm$ are part of the solution, not
independently prescribed boundary data at infinity. The full interface
problem determines their ratio. Multiplication of $m$ by a positive
constant rescales both amplitudes, and we fix this freedom by $m(0)=1$.

\begin{thm}\label{t1.1}
Assume \eqref{1.2}--\eqref{1.4} and \eqref{1.7}, with $d\geq2$.
There exist a positive density $m\in C^{1,\beta}(\mathbb D)$,
normalized by $m(0)=1$, and constants $q_\pm>0$ such that
\begin{equation}\label{1.12}
\begin{cases}
\partial_{ij}(a_{ij}m)-\partial_i(b_i m)=0
    &\text{in }\mathbb D,\\
C^{-1}\leq m(y)\leq C
    &\text{for }y\in\mathbb D,\\
|(m-q_+m_+)(y)|+|\nabla(m-q_+m_+)(y)|
    \leq Ce^{-c|y_1|}
    &\text{for }y_1>1,\\
|(m-q_-m_-)(y)|+|\nabla(m-q_-m_-)(y)|
    \leq Ce^{-c|y_1|}
    &\text{for }y_1<-1.
\end{cases}
\end{equation}
Here $C,c>0$ depend only on $d$, $\beta$, the ellipticity constants,
and the coefficient bounds in the assumptions.
Moreover, every bounded $\mathbb D$-periodic distributional solution
$h$ of $\mathcal L^*h=0$ is a scalar multiple of $m$.
In particular, the normalized positive bounded invariant density is unique.
\end{thm}

\begin{rmk}[The amplitudes and the interface]\label{rmk:amplitudes}
The amplitudes can depend on the coefficients inside the interface
even when the two periodic media coincide. For example, take $A=I$ and
$b(y)=(g(y_1),0,\ldots,0)$ with $g\in C_c^\infty((-1,1))$.
Then $m_\pm=1$ and the normalized density is
\[
m(y)=\exp\left(\int_0^{y_1}g(s)\,ds\right),\qquad
\frac{q_+}{q_-}=\exp\left(\int_{\mathbb R}g(s)\,ds\right).
\]
Thus a fixed interface generally does not allow the two amplitudes
to be assigned independently.
\end{rmk}

The theorem characterizes a stationary density with specified
positivity, boundedness, and asymptotic properties on a noncompact
cylinder. Its uniqueness assertion concerns bounded densities of
either sign, rather than finite-mass probability solutions.
The proof starts with the known periodic densities $m_\pm$ and their
antisymmetric flux potentials. A positive interpolation of $m_\pm$
transforms \eqref{1.1} into a divergence-form adjoint equation whose
remaining drift is supported in the interface layer. We then solve
adjoint Neumann problems on growing finite cylinders and derive bounds
independent of their length. A half-cylinder energy argument yields
exponential stabilization. Finally, a quotient identity and zero
axial flux give uniqueness. All these steps are carried out in
Section~\ref{s2}.

\section{Proof of Theorem \ref{t1.1}}\label{s2}

We first transform the equation using only known periodic densities. A global invariant density is constructed afterwards, so the reduction does not assume the object whose existence is to be proved.

For any $R>0$ and $t\in\mathbb{R}$, denote
$$\mathbb{D}_R:=(-R,R)\times\mathbb{T}^{d-1},\quad
\Sigma_t:=\{t\}\times\mathbb{T}^{d-1}.$$
All the function spaces on the cylinder include $\mathbb{D}$-periodicity. For a matrix $M=(M_{ij})$, we use the convention $(\operatorname{Div}M)_j=\partial_iM_{ij}$, so that
$$\operatorname{div}(M\nabla u)=M_{ij}\partial_{ij}u+\partial_iM_{ij}\partial_ju.$$
Write $\mathcal{L}_\pm u:=a_{\pm,ij}\partial_{ij}u+b_{\pm,i}\partial_i u$. We begin with the periodic flux correctors associated with $m_\pm$ in \eqref{1.6}.

\begin{lemma}\label{lem:periodic-transform}
Under \eqref{1.2}--\eqref{1.4} and \eqref{1.7}, there exist periodic antisymmetric matrices $\phi_\pm\in C^{1,\beta}(\mathbb{Y})$ such that, with
\begin{equation}\label{eq:periodic-current}
c_{\pm,j}=m_\pm b_{\pm,j}-\partial_i(m_\pm a_{\pm,ij}),
\qquad B_\pm=m_\pm A_\pm+\phi_\pm,
\end{equation}
we have
\begin{equation}\label{eq:periodic-potential}
\partial_i\phi_{\pm,ij}=c_{\pm,j},
\qquad
\operatorname{Div}B_\pm=m_\pm b_\pm.
\end{equation}
The matrices $B_\pm$ are bounded and uniformly elliptic, and
\begin{equation}\label{eq:periodic-conjugation}
m_\pm\mathcal{L}_\pm u=\operatorname{div}(B_\pm\nabla u),
\qquad
\mathcal{L}_\pm^*(m_\pm w)=\operatorname{div}(B_\pm^*\nabla w)
\end{equation}
in the distributional sense.
\end{lemma}

\begin{proof}
The cell equation \eqref{1.6} gives $\partial_jc_{\pm,j}=0$, and the centering condition \eqref{1.7} gives $\int_{\mathbb{Y}}c_{\pm,j}=0$. Let $N_{\pm,j}$ be the mean-zero periodic solution of
$$
\Delta N_{\pm,j}=c_{\pm,j}\quad\text{in }\mathbb{Y}.
$$
Then $N_\pm\in C^{2,\beta}(\mathbb{Y})$. The function $\partial_jN_{\pm,j}$ is periodic harmonic and has zero mean, so it vanishes. Set
\begin{equation}\label{eq:canonical-potential}
\phi_{\pm,ij}=\partial_iN_{\pm,j}-\partial_jN_{\pm,i}.
\end{equation}
This gives \eqref{eq:periodic-potential}. The antisymmetry of $\phi_\pm$ and the symmetry of $A_\pm$ imply
$$
\xi\cdot B_\pm\xi=m_\pm\,\xi\cdot A_\pm\xi.
$$
The positive bounds for $m_\pm$ give uniform ellipticity. Expanding $\operatorname{div}(B_\pm\nabla u)$ and using the antisymmetry of $\phi_\pm$ prove the first identity in \eqref{eq:periodic-conjugation}. Testing this identity against $w$ proves the second identity, without differentiating $m_\pm$ twice pointwise.
\end{proof}

Fix $\theta\in C^\infty(\mathbb{R})$ with $0\leq\theta\leq1$, $\theta=0$ on $(-\infty,-1]$, and $\theta=1$ on $[1,\infty)$. Define
\begin{equation}\label{eq:interpolation}
\begin{aligned}
\eta&=(1-\theta)m_-+\theta m_+,\qquad
\phi=(1-\theta)\phi_-+\theta\phi_+,\\
B&=\eta A+\phi,\qquad
r_j=\eta b_j-\partial_iB_{ij}.
\end{aligned}
\end{equation}
All periodic functions in this formula are extended to the whole cylinder. There are constants $\kappa>0$ and $K<\infty$, depending only on the data in Theorem~\ref{t1.1}, such that
\begin{equation}\label{eq:global-transformed-bounds}
\begin{gathered}
K^{-1}\leq\eta\leq K,\qquad
\xi\cdot B\xi\geq\kappa|\xi|^2,\qquad
\|B\|_{C^{1,\beta}(\mathbb{D})}+\|r\|_{C^{0,\beta}(\mathbb{D})}\leq K,\\
B=B_\pm,\quad\eta=m_\pm\quad\text{if }\ \pm y_1>1,
\qquad \operatorname{supp}r\subset[-1,1]\times\mathbb{T}^{d-1}.
\end{gathered}
\end{equation}
Indeed, the symmetric part of $B$ is $\eta A$, and $r=0$ in the two ends by \eqref{eq:periodic-potential}. Expanding the original operator gives
\begin{equation}\label{eq:global-forward-transform}
\eta\mathcal{L}u=\operatorname{div}(B\nabla u)+r\cdot\nabla u.
\end{equation}
For completeness, the calculation in \eqref{eq:global-forward-transform} can be written componentwise. By the antisymmetry of $\phi$ and \eqref{eq:interpolation}, we have

$$
\begin{aligned}
\operatorname{div}(B\nabla u)
&=\partial_i(B_{ij}\partial_j u)\\
&=(\partial_iB_{ij})\partial_j u+B_{ij}\partial_{ij}u\\
&=(\partial_iB_{ij})\partial_j u+
  \eta a_{ij}\partial_{ij}u.
\end{aligned}
$$
Adding $r\cdot\nabla u$ and using the definition of $r_j$ in \eqref{eq:interpolation} yields
$$
\begin{aligned}
\operatorname{div}(B\nabla u)+r\cdot\nabla u
&=\eta a_{ij}\partial_{ij}u+
  (\partial_iB_{ij}+r_j)\partial_j u\\
&=\eta a_{ij}\partial_{ij}u+\eta b_j\partial_j u
=\eta\mathcal{L}u.
\end{aligned}$$

Taking formal adjoints, or testing against a compactly supported smooth function, yields
\begin{equation}\label{eq:global-adjoint-transform}
\mathcal{L}^*(\eta w)=\operatorname{div}(B^*\nabla w-rw).
\end{equation}
Thus the remaining drift $r$ is supported in the interface layer.

Next, we construct approximate invariant measures on finite cylinders by solving the following transformed adjoint Neumann problem. The use of adjoint Neumann problems on growing cylinders is consistent with \cite[Section~4.1]{PetterssonPiatnitski2018}.

\begin{lemma}\label{lem:finite-neumann}
For each $R>4$, there is a nonzero nonnegative weak solution $w_R\in H^1(\mathbb{D}_R)$ of
\begin{equation}\label{eq:finite-neumann}
\begin{cases}
\operatorname{div}(B^*\nabla w_R-rw_R)=0&\text{in }\mathbb{D}_R,\\
(B^*\nabla w_R-rw_R)\cdot n=0&\text{on }\partial\mathbb{D}_R.
\end{cases}
\end{equation}
The solution is strictly positive in $\mathbb{D}_R$, belongs to $C^{1,\beta}(\overline{\mathbb{D}_R})$, and may be normalized by $w_R(0)=1$.
\end{lemma}

\begin{proof}
Fix $R>4$. The only boundary components are the tori $\Sigma_R$ and $\Sigma_{-R}$, with outward unit normals $e_1$ and $-e_1$, respectively. Tangential periodicity produces no lateral boundary terms. Write
$$
H=L^2(\mathbb{D}_R;\mathbb{R}),\qquad V=H^1(\mathbb{D}_R;\mathbb{R}),
\qquad M_r=\|r\|_{L^\infty(\mathbb{D}_R)}.
$$

\medskip\noindent\textbf{Step 1: The forward problem and a coercive shift.}
Consider the forward operator
$$
Pu=\operatorname{div}(B\nabla u)+r\cdot\nabla u,
\qquad (B\nabla u)\cdot n=0\quad\text{on }\partial\mathbb{D}_R.
$$
For a real number $\gamma>0$, define the bilinear form on $V\times V$ by
$$
\mathfrak a_\gamma[u,v]
=\int_{\mathbb{D}_R}B\nabla u\cdot\nabla v
-\int_{\mathbb{D}_R}(r\cdot\nabla u)v
+\gamma\int_{\mathbb{D}_R}uv.
$$
Ellipticity and Young's inequality give
\[
\mathfrak a_\gamma[u,u]\geq
\frac{\kappa}{2}\|\nabla u\|_2^2+
\left(\gamma-\frac{M_r^2}{2\kappa}\right)\|u\|_2^2.
\]
Choose $\gamma\geq1+M_r^2/(2\kappa)$, independently of $R$.
Then $\mathfrak a_\gamma$ is bounded and coercive on the full space
$V$, with a coercivity constant $c_0>0$ independent of $R$.

For every $f\in H$, the Lax--Milgram theorem therefore gives a unique $u\in V$ such that
$$
\mathfrak a_\gamma[u,v]=\int_{\mathbb{D}_R}fv
\qquad(v\in V),\qquad
\|u\|_{H^1}\leq c_0^{-1}\|f\|_2.
$$
We denote this solution by $T_\gamma f$. This is the weak realization of $(\gamma I-P)^{-1}$ with the displayed conormal boundary condition. Indeed, for smooth $u$, integration by parts gives
$$
\mathfrak a_\gamma[u,v]
=\int_{\mathbb{D}_R}(\gamma-P)u\,v
+\int_{\partial\mathbb{D}_R}(B\nabla u\cdot n)v.
$$
Thus allowing all $v\in V$, rather than only tests with zero trace, encodes the homogeneous Neumann condition. The variational construction does not require classical boundary derivatives.

\medskip\noindent\textbf{Step 2: Weak comparison and preservation of nonnegativity.}
Suppose $f\geq0$ almost everywhere and let $u=T_\gamma f$. Put $u^-=\max\{-u,0\}\in V$. On $\{u<0\}$ we have $u=-u^-$ and $\nabla u=-\nabla u^-$; on its complement $u^-=0$ and $\nabla u^-=0$ almost everywhere. Consequently,
$$
0\leq\int_{\mathbb{D}_R}fu^-
=\mathfrak a_\gamma[u,u^-]
=-\mathfrak a_\gamma[u^-,u^-]
\leq-c_0\|u^-\|_{H^1}^2.
$$
Hence $u^-=0$ and $T_\gamma f\geq0$. Applying this argument to the difference of two solutions shows that
$$
f_1\geq f_2\quad\Longrightarrow\quad
T_\gamma f_1\geq T_\gamma f_2.
$$
More generally, if $v\in V$ satisfies $\mathfrak a_\gamma[v,\varphi]\geq0$ for every nonnegative $\varphi\in V$, the same negative-part test gives $v\geq0$. In the classical setting, this hypothesis follows from
$$
(\gamma-P)v\geq0\quad\text{in }\mathbb{D}_R,
\qquad (B\nabla v)\cdot n\geq0\quad\text{on }\partial\mathbb{D}_R.
$$
This is the precise Neumann comparison principle used here. The positive shift is essential for this argument: positive constants solve $Pu=0$ with zero conormal derivative, so the unshifted homogeneous Neumann conditions alone cannot force $u\leq0$. We are not asserting a maximum principle for the adjoint equation by expanding $\operatorname{div}(rw)$.

\medskip\noindent\textbf{Step 3: The constant solution and the $L^\infty$ bound.}
The constant function $\gamma^{-1}$ satisfies
$$
\mathfrak a_\gamma[\gamma^{-1},v]=\int_{\mathbb{D}_R}v
\qquad(v\in V).
$$
Uniqueness therefore gives $T_\gamma1=\gamma^{-1}1$. If $f$ is real and bounded, then comparison and linearity yield
$$
\|T_\gamma f\|_\infty\leq\gamma^{-1}\|f\|_\infty.
$$

\medskip\noindent\textbf{Step 4: Compactness on $L^2$ and bounded nonzero resolvent eigenfunctions.}
Step~1 gives a bounded map $T_\gamma:H\to V$. On the finite cylinder the embedding $V\hookrightarrow H$ is compact by the Rellich theorem. Consequently, regarded as a map $H\to H$, $T_\gamma$ is compact. This compactness is for fixed $R$; it is not compactness on the infinite cylinder. For spectral statements we use the complexification of $H$ and of this real operator.

Let $T_\gamma v=\tau v$, with $\tau\ne0$ and $v\ne0$ possibly complex valued. Then $v=\tau^{-1}T_\gamma v\in H^1$, and
$$
Pv=\mu v,\qquad \mu=\gamma-\tau^{-1}.
$$
In weak form this means
$$
\int_{\mathbb{D}_R}B\nabla v\cdot\nabla\overline\varphi
-\int_{\mathbb{D}_R}(r\cdot\nabla v)\overline\varphi
=-\mu\int_{\mathbb{D}_R}v\overline\varphi
\quad(\varphi\in H^1(\mathbb{D}_R;\mathbb{C})).
$$
Here the natural boundary condition permits tests up to both boundary tori. Now the usual Moser iteration
\cite[Chapter~8]{MR1814364} gives
  $$v\in L^\infty(\mathbb{D}_R),\qquad
\|v\|_\infty\leq C_{R,\mu}\|v\|_2.
$$

\medskip\noindent\textbf{Step 5: Identification of the spectral radius.}
Positivity of the real operator implies, for its complexification,
$$
|T_\gamma f|\leq T_\gamma|f|.
$$
Indeed, $\operatorname{Re}(e^{-it}f)\leq|f|$; apply the real positive operator and then take the supremum over a countable dense set of phases $t$. Thus the $L^\infty$ bound in Step~3 also holds for complex $f$. For the eigenfunction in Step~4, we have
$$
|\tau|\,\|v\|_\infty
=\|T_\gamma v\|_\infty
\leq\gamma^{-1}\|v\|_\infty,
$$
so $|\tau|\leq\gamma^{-1}$. Since $T_\gamma1=\gamma^{-1}1$, the reverse inequality for the spectral radius is automatic. Therefore
$$
\operatorname{spr}(T_\gamma)=\gamma^{-1}>0,
\qquad
\operatorname{spr}(T):=\sup\{|\tau|:\tau\in\sigma(T)\}.
$$

\medskip\noindent\textbf{Step 6: A nonnegative eigenvector of the adjoint.}
The objective of this step is to find a nonzero nonnegative function $w_R$ such that $T_\gamma^*w_R=\gamma^{-1}w_R$. We explain the Hilbert-space and order-theoretic ingredients of this assertion before applying the Krein--Rutman theorem.

\smallskip\noindent\emph{(a) The $L^2$ adjoint and the meaning of positivity.}
We work first in the real Hilbert space
$$
H=L^2(\mathbb{D}_R;\mathbb{R}),\qquad
\langle f,g\rangle=\int_{\mathbb{D}_R}fg.
$$
The adjoint $T_\gamma^*:H\to H$ is characterized by
$$
\langle T_\gamma^*f,g\rangle
=\langle f,T_\gamma g\rangle\qquad(f,g\in H).
$$
Here positivity means preservation of nonnegativity: $S$ is positive if $f\geq0$ almost everywhere implies $Sf\geq0$ almost everywhere. In Step~2, we established this property for $T_\gamma$.

Fix $f\geq0$. For every $g\geq0$, we have $T_\gamma g\geq0$, and hence
$$
\langle T_\gamma^*f,g\rangle
=\langle f,T_\gamma g\rangle
=\int_{\mathbb{D}_R}f\,T_\gamma g\geq0,
$$
which shows the positivity of $T_\gamma^*$.

\smallskip\noindent\emph{(b) The nonnegative cone is closed and generating.}
Define
$$
H_+=\{f\in H:f\geq0\text{ almost everywhere}\}.
$$
It is a cone because $af+bg\in H_+$ whenever $f,g\in H_+$ and $a,b\geq0$. It is convex, closed, and $H_+\cap(-H_+)=\{0\}$.

The word generating means that every vector in $H$ is a difference of two vectors in $H_+$. Indeed, set
$$
f^+=\max\{f,0\},\qquad f^-=\max\{-f,0\}.
$$
Then
$$
f=f^+-f^-,\qquad H=H_+-H_+.
$$
In particular, $H_+$ is a total cone: the closure of $H_+-H_+$ is the whole space, as required by the version of Krein--Rutman used below.

\smallskip\noindent\emph{(c) The dual cone is the same nonnegative cone.}
The continuous dual $H^*$ consists of bounded linear functionals on $H$. Its dual cone is
$$
H_+^*:=\{\ell\in H^*:\ell(g)\geq0\text{ for every }g\in H_+\}.
$$
The Riesz representation theorem identifies each $\ell\in H^*$ with a unique $h\in H$ through
$$
\ell(g)=\langle h,g\rangle=\int_{\mathbb{D}_R}hg
\qquad(g\in H).
$$
 Under this identification the dual cone consists of exactly those $h\in H$ for which $\int hg\geq0$ for all $g\geq0$.

 This explains precisely what is meant by the statement that the cone is self-dual.

\smallskip\noindent\emph{(d) The adjoint is compact.}
Step~4 proved that $T_\gamma$ is compact on $H$. The compact-adjoint theorem implies that $T_\gamma^*$ is compact as well; see \cite{Schaefer1974}.

\smallskip\noindent\emph{(e) The precise Krein--Rutman statement and its hypotheses.}
The version needed here says: if $X$ is a real Banach space with a closed total cone $K$, and $S:X\to X$ is a compact linear operator satisfying $S(K)\subset K$ and $\operatorname{spr}(S)>0$, then there is $w\in K\setminus\{0\}$ with
$$
Sw=\operatorname{spr}(S)w.
$$
The theorem also gives a nonnegative eigenvector of the dual operator in the dual cone. See \cite[Theorem~1.1]{Du2006} and \cite{Schaefer1974} for example.

Apply the theorem with $X=H$, $K=H_+$, and $S=T_\gamma^*$. Parts~(a)--(c) established the order properties, and part~(d) established compactness. Finally, an operator and its Hilbert-space adjoint have conjugate spectra after complexification, so Step~5 yields
$$
\operatorname{spr}(T_\gamma^*)
=\operatorname{spr}(T_\gamma)=\gamma^{-1}>0.
$$
All hypotheses have now been checked. Therefore there exists $w_R\in H_+\setminus\{0\}$ such that
$$
T_\gamma^*w_R=\gamma^{-1}w_R.
$$

\smallskip\noindent\emph{(f) What has been obtained at this stage.}
The conclusion so far is only that $w_R\in L^2$, $w_R\geq0$ almost everywhere, and $w_R\not\equiv0$. It is still an $L^2$ equivalence class, so a value such as $w_R(0)$ is not yet available.

\medskip\noindent\textbf{Step 7: The variational adjoint and its boundary condition.}
For $f\in H$, the transposed Lax--Milgram problem has a unique
$W_f\in V$ such that
\[
\mathfrak a_\gamma[v,W_f]=\langle v,f\rangle
\qquad(v\in V).
\]
For $g\in H$, take $v=T_\gamma g$. The defining equation for
$T_\gamma$ gives
\[
\langle g,W_f\rangle
=\mathfrak a_\gamma[T_\gamma g,W_f]
=\langle T_\gamma g,f\rangle
=\langle g,T_\gamma^*f\rangle.
\]
Thus $W_f=T_\gamma^*f$, and $T_\gamma^*$ maps $H$ into $V$.
In particular, $w_R=\gamma T_\gamma^*w_R\in V$, and
\[
\mathfrak a_\gamma[v,w_R]=\gamma\langle v,w_R\rangle
\qquad(v\in V).
\]
After cancellation of the zero-order terms, this becomes
\[
\int_{\mathbb D_R}(B^*\nabla w_R-rw_R)\cdot\nabla v=0
\qquad(v\in H^1(\mathbb D_R)).
\]
This is precisely the weak formulation of \eqref{eq:finite-neumann},
including its homogeneous flux condition on both boundary tori.

\medskip\noindent\textbf{Step 8: Interior positivity and regularity.}
We have obtained a nonnegative, nonzero weak solution of
$$
\operatorname{div}(B^*\nabla w_R-rw_R)=0.
$$
This has the structure covered by the divergence-form Harnack inequality: the principal matrix $B^*$ is bounded and uniformly elliptic, and $-r$ is a bounded coefficient multiplying the solution inside the flux. For every ball $B_{4s}(y)\Subset\mathbb{D}_R$, in a local periodic chart, the Harnack inequality \cite[Theorem~8.20]{MR1814364} yields
$$
\mathop{\mathrm{ess\,sup}}_{B_s(y)}w_R
\leq C\mathop{\mathrm{ess\,inf}}_{B_s(y)}w_R.
$$
Now, write the equation above as
$$
\operatorname{div}(B^*\nabla w_R)=\operatorname{div}(rw_R).
$$
Interior local boundedness and H\"older regularity \cite[Theorems~8.17 and~8.24]{MR1814364} first give $w_R\in C^{0,\alpha}_{\mathrm{loc}}$ for some $0<\alpha<\beta$. Hence $rw_R\in C^{0,\alpha}_{\mathrm{loc}}$. Then interior divergence-form Schauder regularity \cite[Section~8.11]{MR1814364} gives $w_R\in C^{1,\alpha}_{\mathrm{loc}}$ and therefore a second Schauder application yields $w_R\in C^{1,\beta}_{\mathrm{loc}}$.

Now, if $w_R$ vanishes at an interior point, the Harnack inequality \cite[Theorem~8.20]{MR1814364} would force it to vanish in a neighborhood. A chain of overlapping balls in the connected cylinder would then give $w_R\equiv0$, which is a contradiction. Thus $w_R>0$ throughout $\mathbb{D}_R$.

\medskip\noindent\textbf{Step 9: Boundary regularity and normalization.}
Since $R>4$ and $r$ is supported in $\{|y_1|\leq1\}$, it vanishes in neighborhoods of both boundary tori. In those neighborhoods we have
$$
\operatorname{div}(B^*\nabla w_R)=0,
\qquad (B^*\nabla w_R)\cdot n=0.
$$
The boundaries are smooth and flat in local coordinates, and $B^*\in C^{1,\beta}$. The boundary operator is oblique: its derivative vector is $Bn$, and
$$
(Bn)\cdot n=n\cdot Bn\geq\kappa>0.
$$
The boundary regularity theory for homogeneous conormal problems \cite{Lieberman2013} therefore gives $C^{1,\beta}$ regularity up to each torus. Combined with the interior result and a finite covering of $\overline{\mathbb{D}_R}$, this gives $w_R\in C^{1,\beta}(\overline{\mathbb{D}_R})$. The zero conormal condition consequently has its usual continuous-trace meaning as well.

Finally, $0\in\mathbb{D}_R$ and $w_R(0)>0$. Replacing $w_R$ by $w_R/w_R(0)$ preserves the equation, boundary condition, nonnegativity, and regularity, and gives the required normalization $w_R(0)=1$.

The norm bounds needed only for this fixed-domain construction may depend on $R$. The estimates independent of $R$ are proved separately in Lemma~\ref{lem:uniform-bounds}.
\end{proof}

\begin{lemma}\label{lem:uniform-bounds}
The normalized solutions $w_R$ in Lemma~\ref{lem:finite-neumann} satisfy
\begin{equation}\label{eq:uniform-bounds}
C^{-1}\leq w_R\leq C\quad\text{in }\mathbb{D}_R,
\qquad
\int_{\mathbb{D}_R}|\nabla w_R|^2\leq C,
\end{equation}
where $C$ is independent of $R>4$.
\end{lemma}

\begin{proof}
The Harnack inequality \cite[Theorem~8.20]{MR1814364} in a fixed cylinder containing $\overline{\mathbb{D}_2}$, together with $w_R(0)=1$, gives
$$
C^{-1}\leq w_R\leq C\quad\text{on }\overline{\mathbb{D}_2}.
$$
The constants are uniform because $B$ and $r$ are fixed and satisfy \eqref{eq:global-transformed-bounds}.

In $(2,R)\times\mathbb{T}^{d-1}$, $r=0$ and $w_R$ solves $\operatorname{div}(B_+^*\nabla w_R)=0$, with zero conormal derivative on $\Sigma_R$. Let $M_R=\max_{\Sigma_2}w_R$. Testing this equation with $(w_R-M_R)^+$, which has zero trace on $\Sigma_2$, gives
$$
\kappa\int_{(2,R)\times\mathbb{T}^{d-1}}
|\nabla(w_R-M_R)^+|^2=0.
$$
Hence $w_R\leq M_R$. The corresponding test for the minimum gives
$$
\min_{\Sigma_2}w_R\leq w_R\leq\max_{\Sigma_2}w_R
\quad\text{in }(2,R)\times\mathbb{T}^{d-1}.
$$
The identical argument on $(-R,-2)\times\mathbb T^{d-1}$ treats the left end. This proves the first part of \eqref{eq:uniform-bounds}.

For the energy estimate, choose a scalar cutoff $\chi(y_1)$ that vanishes for $y_1\leq3/2$ and equals $1$ for $y_1\geq2$. Testing the equation \eqref{eq:finite-neumann} by $\varphi=\chi^2w_R$ and using $r=0$ on the support of this test function give
$$
\kappa\int_{\mathbb{D}_R}\chi^2|\nabla w_R|^2
\leq2\left|\int_{\mathbb{D}_R}\chi w_R
B^*\nabla w_R\cdot\nabla\chi\right|
\leq\frac\kappa2\int\chi^2|\nabla w_R|^2
+C\int|\nabla\chi|^2w_R^2.
$$
The last integral is bounded independently of $R$, since it is supported in a fixed strip. This bounds the energy on the right end. An analogous cutoff treats the left end, and an interior Caccioppoli estimate treats the fixed central cylinder.

Consequently, the second part of \eqref{eq:uniform-bounds} follows.
\end{proof}

Local elliptic estimates and \eqref{eq:uniform-bounds} give a subsequence, still denoted by $w_R$, converging in $C^{1,\alpha}$ on compact subsets, for each $0<\alpha<\beta$, to a function $w$. The limit
function $w$ satisfies
\begin{equation}\label{eq:global-w}
\operatorname{div}(B^*\nabla w-rw)=0,
\quad C^{-1}\leq w\leq C,
\quad w(0)=1,
\quad \nabla w\in L^2(\mathbb{D}).
\end{equation}
The energy statement follows from the uniform bound and lower semicontinuity on increasing compact cylinders. Local estimates with global coefficient bounds also give $w\in C^{1,\beta}(\mathbb{D})$. Now, define
\begin{equation}\label{eq:constructed-density}
m=\frac{\eta w}{\eta(0)}.
\end{equation}
Then $m(0)=1$, $m$ has uniform positive upper and lower bounds, and \eqref{eq:global-adjoint-transform} gives $\mathcal{L}^*m=0$.

To investigate the decay properties of the invariant measure constructed above, we first prove the following result on a half-cylinder.

\begin{lemma}\label{lem:tail-stabilization}
Let $D$ be a bounded uniformly elliptic matrix field defined on $\mathbb{Q}=(1,\infty)\times\mathbb{T}^{d-1}$, with a uniform $C^{0,\beta}$ norm. If $u$ is a bounded $\mathbb{D}$-periodic weak solution of
\begin{equation}\label{eq:tail-equation}
\operatorname{div}(D\nabla u)=0\quad\text{in }\mathbb{Q},
\end{equation}
then there is a constant $q\in\mathbb{R}$ such that
\begin{equation}\label{eq:tail-estimate}
|u(y)-q|+|\nabla u(y)|
\leq C\|u\|_{L^\infty(\mathbb{Q})}e^{-c y_1}
\quad(y_1\geq 2).
\end{equation}
The constants depend only on the ellipticity, the coefficient bound and H\"older norm, $\beta$, and $d$.
\end{lemma}

\begin{proof}
Put $M=\|u\|_{L^\infty(\mathbb{Q})}$. Choose a cutoff depending on $y_1$ that is $1$ on $[2,N]$, vanishes near $1$ and beyond $N+1$, and has uniformly bounded derivative. Now, Caccioppoli's inequality gives
$$
\int_{(2,N)\times\mathbb{T}^{d-1}}|\nabla u|^2\leq CM^2.
$$
Letting $N\to\infty$ yields finite tail energy.
For $\psi\in C_c^\infty((1,\infty))$, use $\psi(y_1)$ as a test
function in \eqref{eq:tail-equation}. This gives
\[
\int_1^\infty\left(\int_{\Sigma_t}(D\nabla u)_1\right)
                 \psi'(t)\,dt=0.
\]
Consequently,
$$
J(t)=\int_{\Sigma_t}D_{1i}\partial_i u
$$
is constant. Since $|J(t)|^2\leq C\int_{\Sigma_t}|\nabla u|^2$ and the latter quantity is integrable over $(2,\infty)$, the constant must vanish:
\begin{equation}\label{eq:tail-zero-flux}
J(t)=0.
\end{equation}

Define
$$
E(t)=\int_{(t,\infty)\times\mathbb{T}^{d-1}}|\nabla u|^2,
\qquad \bar u(t)=\int_{\mathbb{T}^{d-1}}u(t,y')\,dy'.
$$
To proceed, finite energy gives a sequence $T_k\to\infty$ with $\int_{\Sigma_{T_k}}|\nabla u|^2\to0$. Testing the equation \eqref{eq:tail-equation} by $u$ and integration by parts on $(t,T_k)\times\mathbb{T}^{d-1}$, followed by this limit, gives
$$
\int_{(t,\infty)\times\mathbb{T}^{d-1}}
D\nabla u\cdot\nabla u
=-\int_{\Sigma_t}u(D\nabla u)_1.
$$
The outer boundary term vanishes because $u$ is bounded and the cross-sectional gradient norm tends to zero along the chosen sequence. Using \eqref{eq:tail-zero-flux} and the Poincar\'e inequality on the torus, we obtain for almost every $t>2$
\begin{equation}\label{eq:tail-energy-ode}
\begin{aligned}
\kappa E(t)
&\leq-\int_{\Sigma_t}(u-\bar u(t))D_{1i}\partial_i u\\
&\leq C\|u-\bar u(t)\|_{L^2(\Sigma_t)}
\|\nabla u\|_{L^2(\Sigma_t)}\\
&\leq C\int_{\Sigma_t}|\nabla u|^2=-CE'(t),
\end{aligned}
\end{equation}
which, due to Gronwall's Lemma, further implies that
\begin{equation}\label{eq:tail-energy-decay}
E(t)\leq CM^2e^{-ct}\qquad(t\geq2).
\end{equation}
For $y_1\geq3$, interior gradient estimates and the Poincar\'e
inequality on a fixed-radius ball give
\[
|\nabla u(y)|
\leq C\left(\int_{B_{1/4}(y)}|\nabla u|^2\right)^{1/2}
\leq CMe^{-c y_1},
\]
after decreasing $c$ if necessary. Uniform interior estimates give
the same bound, with a larger $C$, for $2\leq y_1\leq3$.
It remains to prove the other part of \eqref{eq:tail-estimate}. Fix $y'\in\mathbb{T}^{d-1}$. For $s>t\geq3$, the fundamental theorem of calculus gives
\begin{equation}\label{2.18}
\begin{aligned}
|u(s,y')-u(t,y')|
&=\left|\int_t^s\partial_1u(\tau,y')\,d\tau\right|\\
&\leq C_1M\int_t^se^{-c\tau}\,d\tau
=\frac{C_1M}{c}(e^{-ct}-e^{-cs})
\leq\frac{C_1M}{c}e^{-ct}.
\end{aligned}\end{equation}
The bound is independent of $y'$ and tends to zero as $t\to\infty$. Thus $u(t,\cdot)$ is uniformly Cauchy on the torus and converges uniformly to a continuous function $U(y')$.  Letting $s\to\infty$ in \eqref{2.18} gives
\begin{equation}\label{2.19}
\sup_{y'\in\mathbb{T}^{d-1}}|u(t,y')-U(y')|
\leq\frac{C_1M}{c}e^{-ct}.\end{equation}

To proceed,
take any $y',z'\in\mathbb{T}^{d-1}$. They can be joined by a shortest path $\Gamma$ in the flat torus whose length is at most
$L_d=\sqrt{d-1}/2$. Integrating the tangential derivative along this path, at fixed $t$, gives
$$
\begin{aligned}
|u(t,y')-u(t,z')|
&\leq\int_\Gamma|\nabla_{y'}u(t,\zeta')|\,d\ell(\zeta')\\
&\leq L_d\sup_{\Sigma_t}|\nabla_{y'}u|
\leq L_dC_1Me^{-ct}.
\end{aligned}
$$
Passing to the limit $t\to\infty$ yields $|U(y')-U(z')|=0$. Since the two points are arbitrary, $U$ is a single constant; denote it by $q\in\mathbb{R}$.
Using $U\equiv q$ in \eqref{2.19}, we obtain
$$
|u(t,y')-q|\leq\frac{C_1M}{c}e^{-ct}.
$$
This proves the lemma.
\end{proof}

\noindent\textbf{Proof of Theorem \ref{t1.1}}: We are ready to prove the decay estimates and uniqueness in Theorem \ref{t1.1}.

Apply Lemma~\ref{lem:tail-stabilization} to $w$ on the right end with $D=B_+^*$, and on the left end after reflection of the axial variable. There are constants $a_\pm>0$ such that $w-a_\pm$ and their gradients decay exponentially in the corresponding ends. Positivity of $a_\pm$ follows from the uniform lower bound in \eqref{eq:global-w}. In view of \eqref{eq:constructed-density}, set
$$
q_\pm=\frac{a_\pm}{\eta(0)}.
$$
Since $m=\eta w/\eta(0)$, with $\eta=m_+$ for $y_1>1$ and
$\eta=m_-$ for $y_1<-1$, the two asymptotic estimates \eqref{1.12} follow. This proves existence and all asserted estimates.

The original adjoint current is
\begin{equation}\label{eq:physical-current}
F_i=\partial_j(a_{ij}m)-b_i m,
\qquad \partial_iF_i=0,\quad \text{in }\mathbb{D}.
\end{equation}

For any $R_2>R_1$, integrating the equation \eqref{1.6} over $(R_1,R_2)\times \mathbb{T}^{d-1}$ after using the $\mathbb{D}$-periodicity of $m_\pm$ yields that
\begin{equation*}\begin{aligned}
\int_{\Sigma_{R_2}}\partial_1(a_{+,11}m_+)-b_{+,1}m_+=
\int_{\Sigma_{R_1}}\partial_1(a_{+,11}m_+)-b_{+,1}m_+,
\end{aligned}\end{equation*}
which implies that for any $R\in (-\infty,+\infty)$, the integral
\begin{equation*}
\begin{aligned}
\int_{\Sigma_{R}}\partial_1(a_{+,11}m_+)-b_{+,1}m_+=0,
\end{aligned}\end{equation*}
due to
\begin{equation*}\begin{aligned}
\int_{\mathbb{T}^{d}}\partial_1(a_{+,11}m_+)-b_{+,1}m_+=0,
\end{aligned}\end{equation*}
where we have used the centering condition \eqref{1.7} in the equality above.

The same argument shows that the integral $\int_{\Sigma_t}F_1$ is a constant in $t$. And the exponential asymptotics of $m$ and its gradient therefore give
\begin{equation}\label{eq:physical-zero-flux}
\int_{\Sigma_t}F_1=0\qquad(t\in\mathbb R).
\end{equation}

It remains to prove uniqueness. Let $h$ be any bounded $\mathbb{D}$-periodic distributional solution of $\mathcal{L}^*h=0$. We first justify the local regularity of $h$, including its membership in the weak solution space, without assuming a priori that $h$ has a weak gradient. The proof is due to \cite[Proposition A.4]{2022Green} for the case of $b=0$. We recall the main ideas and point out the modification needed to handle the drift term.

Now, assume $h$ is bounded and satisfies 
\begin{equation}\label{2.20}
\partial_{ij}(a_{ij}h)-\partial_{i}(b_{i}h)=0\quad \text{in }B_{2r},\end{equation}
where $B_{2r}$ is a standard ball in $\mathbb{R}^d$ with center $0$ and radius $2r$. Without loss of generality, we may suppose that ${A}(0)=(a_{ij}(0)) = I_d$ and $r = 1$. Let $\{\phi_\delta\}_{\delta>0}$ be the standard radial mollifier. As pointed out in \cite[Proposition A.4]{2022Green}, the estimate
\begin{equation}\label{2.21}
\left|h * \phi_{2\delta}(0) - h * \phi_\delta(0)\right|
\le C\delta^{\alpha-d}
\|h\|_{{L}^1(B_{2\delta})},\quad \forall\delta\in (0,1/8),
\end{equation}
shows that $h\in C^{0,\alpha}_{\text{loc}}$ by translation.
Let $\Gamma$ be the fundamental solution of $-\Delta$ in
$\mathbb{R}^d$, and define
$$
w_\delta=\Gamma*(\phi_{2\delta}-\phi_\delta).
$$
Then $w_\delta\in C^\infty(\mathbb{R}^d)$ satisfies
$$
-\Delta w_\delta=\phi_{2\delta}-\phi_\delta
\quad\text{in }\mathbb{R}^d.
$$
Since the mollifiers are radially symmetric and have unit mass,
the mean value property for harmonic functions gives
$$
(\Gamma*\phi_{2\delta})(x)
=(\Gamma*\phi_\delta)(x)
=\Gamma(x)
\qquad\text{for }|x|>2\delta.
$$
Consequently, $w_\delta$ vanishes outside $B_{2\delta}$. Moreover, for every $\sigma\in (0,1]$, classical Schauder estimates show that $w_\delta$ satisfies
$$
\|\nabla^2 w_\delta\|_{L^\infty(\mathbb{R}^d)}
\le \|\phi_{2\delta} - \phi_\delta\|_{L^\infty(\mathbb{R}^d)}
+ \delta^\sigma [\phi_{2\delta} - \phi_\delta]_{C^{0,\sigma}(\mathbb{R}^d)}
= C\delta^{-d}.
$$
For each $\delta < 1/2$, testing $w_\delta\in C_c^\infty(B_{2\delta})\subseteq C_c^\infty(B_1)$ in the weak formulation of the adjoint equation \eqref{2.20} yields
$$
\begin{aligned}
\left|h * \phi_{2\delta}(0) - h * \phi_\delta(0)\right|
&= \left|\int_{B_{2\delta}} (\phi_{2\delta}-\phi_\delta)h\,dx\right|\\
&= \left|\int_{B_{2\delta}} \Delta w_\delta h\,dx\right|\\
&\leq \left|\int_{B_{2\delta}} \bigl(\Delta w_\delta-\operatorname{tr}({A}\nabla^2 w_\delta)\bigr)h\,dx\right|+
\left|\int_{B_{2\delta}} b_i\partial_i w_\delta h\,dx\right|\\[8pt]
&\le C\|h\|_{L^1(B_{2\delta})}\|\nabla^2 w_\delta\|_{L^\infty(B_{2\delta})}\left(\|{A}-I_d\|_{L^\infty(B_{2\delta})}+\delta\|b\|_{L^\infty(B_1)}\right)\\[8pt]
&\le C\left([{A}]_{C^{0,\alpha}(B_1)}+\|b\|_{L^\infty(B_1)}\right)\delta^{\alpha-d}\|h\|_{L^1(B_{2\delta})}.
\end{aligned}
$$
This yields \eqref{2.21}. Writing the equation in divergence form and applying Schauder
estimates for divergence-form equations, we obtain
$$h\in C^{1,\beta}_{\mathrm{loc}}(\mathbb{D}).$$

Set $z=h/m$. Since $\mathcal{L}^*h=\mathcal{L}^*m=0$, a direct computation shows that
the function $z$ satisfies
\begin{equation}\label{eq:uniqueness-ratio}
\partial_i\left(ma_{ij}\partial_j z\right)
+\left[\partial_j(a_{ij}m)-b_i m\right]\partial_i z=0
\quad\text{in }\mathbb{D}.
\end{equation}
Since $A=A^*$ and $m>0$, expanding the divergence and
dividing by $m$ gives 
$$
a_{ij}\partial_{ij}z
+\left[\frac{2}{m}\partial_j(a_{ij}m)-b_i\right]
\partial_i z=0
\quad\text{in }\mathbb{D}.
$$
Arguing as in the derivation of \eqref{eq:uniqueness-ratio} and using \eqref{eq:periodic-potential}, we find
that $z_+=h/m_+$ satisfies
\begin{equation*}\begin{aligned}
&\partial_i\left(m_+a_{+,ij}\partial_j z_+\right)
+\left[\partial_j(a_{+,ij}m_+)-b_{+,i} m_+\right]\partial_i z_+\\[6pt]
=&\partial_i\left(m_+a_{+,ij}\partial_j z_+\right)-\partial_j\phi_{+,ji}\partial_i z_+\\[6pt]
=&\partial_i\left(\left(m_+a_{+,ij}-\phi_{+,ij}\right)\partial_j z_+\right)=0
\quad\text{in }(1,+\infty)\times \mathbb{T}^{d-1}.
\end{aligned}\end{equation*}
All equations above are understood in the distributional sense. Therefore, in each end, $h/m_\pm$ is a bounded solution of the corresponding homogeneous divergence-form equation. Lemma~\ref{lem:tail-stabilization} thus gives constants $p_\pm\in\mathbb{R}$ such that $h-p_\pm m_\pm$ and their gradients decay exponentially. Recall that $m-q_\pm m_\pm$ and their gradients decay exponentially. Hence, due to \eqref{eq:physical-current} and \eqref{eq:uniqueness-ratio}, $z=h/m$ is bounded and satisfies
\begin{equation}\label{2.26}
\operatorname{div}(mA\nabla z)+F\cdot\nabla z=0\text{ in }\mathbb{D},
\qquad
z\to\frac{p_\pm}{q_\pm},\quad
|\nabla z|\leq Ce^{-c|y_1|}
\end{equation}
in the respective ends. Here $\nabla z\in L^2(\mathbb{D})$ and $F$ is bounded and divergence free. Testing \eqref{2.26} with $z$ on $\mathbb{D}_R$ gives
\begin{equation}\label{eq:uniqueness-energy}
\int_{\mathbb{D}_R}mA\nabla z\cdot\nabla z
=\int_{\partial\mathbb{D}_R}z(mA\nabla z)\cdot n
+\frac12\int_{\partial\mathbb{D}_R}z^2F\cdot n.
\end{equation}
The first boundary integral tends to zero by the decay of $\nabla z$. In the second, subtract the square of the corresponding limiting constant on each boundary torus; the subtracted terms vanish by \eqref{eq:physical-zero-flux}, and the remainders tend to zero exponentially. Letting $R\to\infty$ in \eqref{eq:uniqueness-energy} gives $\nabla z=0$. Thus $h$ is a scalar multiple of $m$, and the normalization $m(0)=1$ proves the uniqueness asserted in Theorem~\ref{t1.1}.

\begin{center}{\textbf{Acknowledgements}}
\end{center}

The work of Y. Zhang was supported by the National Natural Science Foundation of China under Grant 12401256 and by the
Hubei Provincial Natural Science Foundation of China under Grant 2024AFB357.

The authors also thank OpenAI for suggesting the idea of removing the assumption $b_1\equiv0$ imposed in an earlier version of this paper.\\

\noindent \textbf{Declarations}\\

\noindent\textbf{Ethical Approval}: Not applicable.\\
\noindent \textbf{Conflict of interest}: On behalf of all authors, the corresponding author states that there is no conflict of interest.\\
\textbf{Availability of data and materials}: Data sharing is not applicable to this article as no datasets were generated or analysed during the current study.\\

\bibliographystyle{abbrvnat}

\end{document}